\UseRawInputEncoding

\documentclass[12pt]{amsart}
\usepackage{amsfonts}
\usepackage{amsmath}
\usepackage{amssymb,latexsym}
\usepackage[mathcal]{eucal}
\usepackage[pdftex,bookmarks,colorlinks,breaklinks]{hyperref}

\input xy
\xyoption{all}

\newcommand{\Acal}{\mathcal{A}}
\newcommand{\Bcal}{\mathcal{B}}

\newcommand{\Gcal}{\mathcal{G}}

\newcommand{\Ncal}{\mathcal{N}}

\newcommand{\Pcal}{\mathcal{P}}
\newcommand{\Qcal}{\mathcal{Q}}
\newcommand{\Rcal}{\mathcal{R}}
\newcommand{\Scal}{\mathcal{S}}

\newcommand{\Xcal}{\mathcal{X}}
\newcommand{\Ycal}{\mathcal{Y}}
\newcommand{\Zcal}{\mathcal{Z}}

\newcommand{\ch}{\mathbf{1}}

\newcommand{\Z}{\mathbb{Z}}

\newcommand{\R}{\mathbb{R}}

\newcommand{\N}{\mathbb{N}}

\newcommand{\E}{\mathbb{E}}

\newcommand{\Ab}{\mathbf{A}}
\newcommand{\Bb}{\mathbf{B}}

\newcommand{\Lb}{\mathbf{L}}
\newcommand{\Mb}{\mathbf{M}}

\newcommand{\Pb}{\mathbf{P}}
\newcommand{\Qb}{\mathbf{Q}}

\newcommand{\Tb}{\mathbf{T}}

\newcommand{\Wb}{\mathbf{W}}
\newcommand{\Xb}{\mathbf{X}}
\newcommand{\Yb}{\mathbf{Y}}
\newcommand{\Zb}{\mathbf{Z}}

\newcommand{\al}{\alpha}

\newcommand{\del}{\delta}

\newcommand{\ep}{\epsilon}
\newcommand{\sig}{\sigma}

\newcommand{\la}{\lambda}

\newcommand{\ol}{\overline}

\newcommand{\br}{\vspace{3 mm}}

\newcommand{\tri}{\bigtriangleup}
\newcommand{\rest}{\upharpoonright}

\newcommand{\Aut}{{\rm{Aut\,}}}

\newcommand{\Id}{{\rm{Id}}}

\swapnumbers
\theoremstyle{plain}
\newtheorem{thm}{Theorem}[section]
\newtheorem{cor}[thm]{Corollary}
\newtheorem{lem}[thm]{Lemma}
\newtheorem{prop}[thm]{Proposition}

\newtheorem{claim}{Claim}

\theoremstyle{definition}
\newtheorem{defn}[thm]{Definition}
\newtheorem{defns}[thm]{Definitions}

\newtheorem{rmk}[thm]{Remark}

\newtheorem{rem}[thm]{Remark}

\newtheorem{prob}[thm]{Problem}

\begin{document}


\title[Generic classes of ergodic transformations]
{On some generic classes of ergodic measure preserving transformations}

\author[Eli Glasner]{Eli Glasner}
\address{Department of Mathematics,
Tel-Aviv University, Ramat Aviv, Israel}
\email{glasner@math.tau.ac.il}

\author[Jean-Paul Thouvenot]
{Jean-Paul Thouvenot}
\address{Laboratoire de Probabilit\'es, Universit\'e 
Statistique et Mod\'{e}lisation, Sorbonne Universit\'{e},
4 Place Jussieu, 75252 Paris Cedex 05, France}
\email{jean-paul.thouvenot@upmc.fr}

\author[Benjamin Weiss]{Benjamin Weiss}
\address{Mathematics Institute, Hebrew University of Jerusalem,
Jerusalem, Israel}
\email{weiss@math.huji.ac.il}


\setcounter{secnumdepth}{2}



\setcounter{section}{0}


 
%

\begin{abstract}
We answer positively a question of Ryzhikov, namely we show that 
being a relatively weakly mixing extension is a comeager property in the Polish group
of measure preserving transformations.
We study some related classes of ergodic transformations and their interrelations.
In the second part of the paper we show that for a fixed ergodic $T$ with property $\Ab$, 
a generic extension $\hat{T}$ of $T$ also has the property $\Ab$.
Here $\Ab$ stands for each of the following properties:
(i) having the same entropy as $T$, (ii) Bernoulli, (iii)  K, and (iv) loosely Bernoulli.
\end{abstract}

\keywords{relative weak mixing, comeager properties, prime dynamical systems,
Bernoulli systems, K-systems, loosely Bernoulli systems}

\thanks{{References: 46 entries. UDK: 517.987.
\em 2010 MSC2010:
37A25, 37A05, 37A15, 37A20}}

\begin{date}
{August 20, 2021}
\end{date}

\maketitle


\tableofcontents
\setcounter{secnumdepth}{2}


\setcounter{section}{0}


\section*{Introduction}

Motivated by a question of Valery Ryzhikov \cite{R-06}, regarding the nature of the class of ergodic transformations
$\Xb = (X, \Xcal, \mu, T)$ which admit a proper factor 
$\Xb \to \Yb$ with $\Yb = (Y, \Ycal, \nu, T)$ nontrivial and where the extension is 
relatively weakly mixing (we call this class RWM), we consider in this note the
RWM property and some related classes as follows:

\br

We view the collection MPT (measure preserving transformations) 
of invertible measure preserving transformations
on a nonatomic standard probability space $(X, \Xcal,\mu)$, as the group $\Aut(\mu)$
of automorphisms of $(X, \Xcal,\mu)$.
The topology on $\Aut(X,\mu)$ is induced by a complete metric
$$
D(S,T) = \sum_{n \in \N} 2^{-n}  (\mu(SA_n \tri TA_n) + \mu(S^{-1}A_n \tri T^{-1}A_n)), 
$$
with $\{A_n\}_{n \in \N}$ a dense sequence in the measure algebra
$(\mathcal{X},d_\mu)$, where $d_\mu(A,B) = \mu(A \tri B)$.
Equipped with this topology $\Aut(X, \mu)$ is a Polish topological group.
The set ${\rm MPT}_e \subset \Aut(\mu)$,
comprising the ergodic transformations, is a dense $G_\del$ subset of $\Aut(\mu)$.
Thus $MPT = \Aut(\mu)$ and we use the latter when we want to emphasise the group
structure of this space,

\begin{defns}\label{def}
Given an ergodic system $\Xb$ we say that a factor map $\pi : \Xb \to \Yb$ is
{\em nontrivial}
when $\Yb$ is not the trivial one point system and $\pi$ is not an isomorphism.
(We often refer to a factor map $\pi : \Xb \to \Yb$ from $\Xb$ to $\Yb$
also as an extension of $\Yb$.)
An ergodic dynamical system $\Xb = (X, \Xcal, \mu, T)$ is :
\begin{enumerate}
\item 
RD ({\em relatively distal}) if there exists a nontrivial factor map $\pi : \Xb \to \Yb$ 
such that the extension is relatively distal.
Every ergodic distal system which is not isomorphic to a cyclic permutation on $\Z_p$
with $p$ a prime number, is RD. 
\item
RWM ({\em relatively weakly mixing}), if there exists a nontrivial factor map $\pi : \Xb \to \Yb$ 
such that the extension is relatively weakly mixing; i.e. such that the relative product 
$\Xb \underset{\Yb}{\times} \Xb$ is ergodic. 
By the Furstenberg-Zimmer theorem, every ergodic system which is neither distal nor weakly mixing is RWM.
\item
TRD ({\em totally relatively distal}), 
if every nontrivial factor map $\pi : \Xb \to \Yb$ is relatively distal. Every distal system is TRD.
Every prime system is (vacuouesly) TRD.
\item
TRD$^\sim$
if it is TRD, but not prime. 
Every distal system which is not isomorphic to $\Z_p$ for some prime number $p$,  is TRD$^\sim$.
By Veech's theorem every nonprime weakly mixing simple system is in TRD$^\sim$
(see \cite[Theorem 12.3]{G}).
The example in \cite{dJ} is a weakly mixing system in TRD$^\sim$ with no prime factor.
\item
TRWM ({\em totally relatively weakly mixing}),
if every factor map $\pi : \Xb \to \Yb$ which is nontrivial 
is relatively weakly mixing. 
Every prime system is (vacuously) TRWM,
and probably also products of disjoint prime systems are TRWM,
but it is not clear to us what else is there in this class.
 An ergodic system $\Xb$ in TRWM is either weakly mixing, or it 
 admits the finite system $\Z_p = \Z /p\Z$ as a factor, for some prime number $p \ge 2$,
 such that the extension $\Xb \to \Z_p$ is relatively weakly mixing and 
 $\Xb$ is not of the form $\Xb = \Z_p \times \Wb$ for a weakly mixing system $\Wb$.
(An example of a system which has the latter form --- but is not TRWM --- is provided in \cite{O}.)
In fact, by the Furstenberg-Zimmer structure theorem an ergodic system $\Xb \in$ TRWM has the structure
$\Xb \to \Yb$, where $\Yb$ is the maximal distal factor of $\Xb$ and the extension is
relatively weakly mixing. If $\Yb$ is the trivial system then $\Xb$ is weakly mixing.
Otherwise, as $\Xb$ is TRWM, we must have $\Yb = \Z_p$ for some prime $p \ge 2$.
Finally  $\Xb$ can not have the form $\Xb = \Z_p \times \Wb$ for a weakly mixing system $\Wb$,
since if it were the projection map $\Xb \to \Wb$ would not be a relatively weakly mixing extension.
\item
An ergodic system $\Xb$ is called  {\em  $2$-fold quasi-simple}
 ({\em $2$-fold distally simple}) if every ergodic non-product $2$-fold self-joining of it is compact 
 (relatively distal respectively) over the marginals.
 Every simple system (hence every weakly mixing MSJ system) is $2$-fold quasi-simple.
  \end{enumerate}
 \end{defns}

\br

For more details and results concerning these notions see (among other publications)
the following list:
\begin{enumerate}
\item[(i)]
In \cite{Ru} Rudolph introduced the notion of MSJ (minimal self joinings), 
and in  \cite{V} Veech introduced $2$-simple systems.
Simple systems of higher order and their joinings were studied in \cite{dJ-R}.
In \cite{K-92} King has shown that simplicity of order 4 implies simplicity of all orders
and in \cite{GHR} this was improved to showing that order 3 suffices.
\item[(ii)] 
Various generalizations of simplicity were introduced and previous results
were sharpened in  \cite{JLM} and \cite{R}. 
  \item[(iii)]
The first genericity type theorem was King's paper \cite{K-00} 
where he has shown that having roots of all orders is generic.
This was followed by Ageev  who first proved 
that, for every finite abelian group $G$, having $G$ in the centraliser is generic \cite{A-00};
thereby showing that the class of prime transformations is meager,
and then in \cite{A-03}, that the generic system is neither simple nor semisimple.
\item[(iv)]
Stronger and stronger results of this nature are to be found in
\cite{dR-SL}, \cite{S-E}, \cite{R-T}, \cite{D}, \cite{So}, and \cite{C-K}.
\end{enumerate}

For information concerning the Furstenberg-Zimmer theorem 
and structure theory in ergodic theory we refer to \cite{G}.

In Section \ref{pre} we present a preliminary study of the nature of these classes and their interrelations.
In Section \ref{RWMg} we answer positively Ryzhikov's question \cite{R-06}, namely we show that 
the property RWM is comeager.

\br

In the second part of the paper, we change slightly the framework of our discussion.
We mainly fix an ergodic transformation $T$ with a certain property,
and then consider either its family of factors, 
or its family of ergodic extensions.

In Section \ref{thouvenot} we show that 
the relatively weakly mixing factors of any fixed positive entropy $T$ form a dense $G_\del$ set.
In Section \ref{add-0} we show that a generic extension of any ergodic $T$ does not add entropy.
In Sections \ref{bernoulli} we prove that 
for any fixed Bernoulli transformation of finite entropy the generic extension is Bernoulli.
In Section \ref{K-Sec} we show that for any fixed K-automorphism $T$, the generic extension is K,
and it is relatively mixing over $T$.
Finally, in Section \ref{LB}  we show that for any fixed loosely Bernoulli transformation $T$, 
the generic extension is loosely Bernoulli.

\br

We remind the readers that the classes 
of weakly mixing, Rank-1, rigid, $\kappa$-mixing
and zero-entropy systems, 
are all comeager sets in MPT (see e.g. \cite{H}, \cite{K-S} and \cite{St})
\footnote
{
For $\kappa \in [0,1]$, $T$  is said to be {\em {$\kappa$-mixing}}
if there is a sequence $n_k \nearrow \infty$ in $\N$ such that for any two
measurable sets $A, B \in \Xcal$ 
$$
\lim_{k \to \infty}\mu(T^{n_k}A \cap B) = \kappa\mu(A)\mu(B) + (1 - \kappa) \mu(A \cap B).
$$
}.

\br

Note that, whereas by Austin \cite{Au}, every positive entropy ergodic $T$ decomposes
non-trivially as a direct product; Friedman's result \cite{Fr}
that, for $0 < \kappa < 1$, a $\kappa$-mixing system admits no non-trivial product as a factor, combined with the 
fact that $\kappa$-mixing is comeager \cite{St}, show that the generic
(zero-entropy ergodic) $T$ does not split (and moreover can not have a non-trivial product as a factor).

\br

We would like to thank Valery Ryzhikov for bringing his question in \cite{R-06} to our attention and 
for several helpful e-conversations.
We  thank Oleg Ageev for a careful and thorough reading of a previous draft of the paper.
His useful comments and corrections considerably improved our work.
We also thank Tim Austin for pointing out an
inaccuracy in the proof of Theorem \ref{Kthm} in an earlier version.

\br

\section{Some general results concerning the classes mentioned in the introduction}\label{pre}

\begin{claim}\label{RD}
Every ergodic system with positive entropy is RD.
\end{claim}

\begin{proof}
It is well known that every Bernoulli system is RD. The case of a positive entropy ergodic system
follows from the weak Pisnsker theorem \cite{Au}.
\end{proof}

\begin{claim}\label{c0}
The property TRD is inherited by factors.
 \end{claim}
 
 \begin{proof}
 Let $\Xb$ be a TRD system and $\Xb \to \Yb$ a nontrivial factor.
 Suppose $\Yb \to \Zb$ is a nontrivial factor of $\Yb$.
 Let $\Yb \to \Zb_{rd} \to \Zb$ be the  diagram obtained by the Furstenberg-Zimmer 
 structure theorem (the relative version); i.e $\Zb_{rd}$ is the largest distal extension of $\Zb$ in $\Yb$
 and $\Yb \to \Zb_{rd}$ is a relatively weakly mixing extension.
 Now from the combined diagram
 $$
 \Xb \to \Yb \to \Zb_{rd} \to \Zb
 $$
 and the fact that $\Xb$ is TRD, we deduce that the map
 $ \Xb  \to \Zb_{rd}$ is a relatively distal extension, and it follows that the intermediate
 extension $\Yb \to \Zb_{rd}$ is also distal (this is a nontrivial result, see \cite[Theorem 10.18]{G}). 
Thus the extension $\Yb \to \Zb_{rd}$ is both relatively weakly mixing and relatively distal, hence an isomorphism.
 This means that indeed $\Yb \to \Zb$ is a relatively distal extension, as claimed.
 \end{proof}

\br

\begin{claim}\label{c1}
An ergodic $\Xb$ is not RWM iff it is TRD. Thus
$$
{\rm MPT}_e = {\rm RWM}  \ 
\sqcup 
\ {\rm TRD}.
$$
 \end{claim}

\begin{proof}
Suppose $\Xb$ is not RWM and $\Xb \to \Yb$ is a factor map, then
by the Furstenberg-Zimmer structure theorem (the relative version)
there is a canonical diagram $\Xb \to \Yb_{rd} \to \Yb$, where
$\Yb_{rd} \to \Yb$ is relatively distal and $\Xb \to \Yb_{rd}$ is relatively weakly mixing.
However, by assumption the map $\Xb \to \Yb_{rd}$ is an isomorphism, so that
$\Xb = \Yb_{rd} \to \Yb$ is relatively distal. The other inclusion is trivial: 
if $\Xb$ is TRD it can not be RWM.
\end{proof}

\br
 
 \begin{claim}
 Every $2$-fold distally simple system is TRD. 
 A fortiori, Every $2$-fold quasi-simple system is TRD.
  \end{claim}

\begin{proof}
Suppose $\Xb \to \Yb$ is a proper factor of a $2$-fold distally simple system $\Xb$.
Let $\Xb \to \Yb_{rd} \to \Yb$ be the corresponding Furstenberg-Zimmer diagram.
Consider the factor map $\Xb \underset{\Yb_{rd}}{\times} \Xb \to \Xb$
defined by the projection map (on either the first or the second coordinate).
By definition the system $\Xb \underset{\Yb_{rd}}{\times} \Xb$ is ergodic
and it is a non-product self joining of $\Xb$. By the definition of $2$-fold distal simplicity
the extensions $\Xb \underset{\Yb_{rd}}{\times} \Xb \to \Xb$ and hence also the intermediate
extension 
$$
\Xb \underset{\Yb_{rd}}{\times} \Xb \to \Yb_{rd} \underset{\Yb_{rd}}{\times} \Xb \cong \Xb 
$$
are distal extensions.

On the other hand, by \cite[Theorem 9.23]{G}, the extension
$\Xb \underset{\Yb_{rd}}{\times} \Xb \to \Xb$ is weakly mixing, and therefore it is an 
isomorphism. In turn, this means that $\Xb \to \Yb_{rd}$ is an isomorphism,
so that the extension $\Xb \to \Yb$ is relatively distal, as required.
\end{proof}

\br

\br

 \begin{prob}
Is the example in \cite{G-W81} TRD?
 \end{prob}

\begin{prob}
What is the extent of the class TRWM?
\end{prob}

\br

\begin{claim}\label{TRWM0}
Every TRWM system has zero entropy.
\end{claim}

\begin{proof}
By  Sinai's theorem every ergodic positive entropy system, say $\Xb$, admits a Bernoulli factor
$\Xb \to \Bb$.
The Bernoulli system $\Bb$ has a nontrivial factor $\sig : \Bb \to \Zb$, 
such that the extension $\sig$ is a compact group extension.
Now the resulting factor map $\Xb \to \Zb$ is not relatively weakly mixing. 
It follows that every TRWM system has zero entropy.
\end{proof}

\br

By Krieger's theorem every factor of an ergodic system $\Xb \to \Yb$,
with $\Yb$ having entropy $< \log 2$, is determined
by a partition $\al =\{A, X \setminus A\}$ for some $A \in \Xcal$ of positive measure,
so that $\Ycal = \bigvee_{N \in \N} \bigvee_{j = -N}^N T^j \al$.

\br

\begin{prop}\label{RWMdel}
If a dynamical system $\Xb$ is RWM then, for any $\del >0$, it has a factor 
$\pi : \Xb \to \Yb$ such that (i) $\Yb$ is infinite with entropy $ < \del$, and (ii)
the extension $\pi$ is nontrivial and relatively weakly mixing.
In particular, taking $\del = \log 2$, we conclude, by Krieger's theorem, that $\Yb$ can
be taken as a subshift on two symbols; i.e. that $\Yb$ admits a generator of the form
$\{A, A^c\}$.
\end{prop}

\begin{proof}
We give two proofs as follows:

(a)\ 
There is, by definition a factor map $\sig : \Xb \to \Zb$ with $\Zb$ infinite and such that
the extension $\sig$ is relatively weakly mixing.
If $\Zb$ has zero entropy there is nothing to prove. Otherwise there is 
a factor map $\Zb \to \Yb'$ with $\Yb'$ infinite and having entropy $< \del$.
Let $\Zb \to \Yb \to \Yb'$ be the relative Fustenberg-Zimmer tower, so that $\Yb$ is the maximal
distal extension of $\Yb'$ within $\Zb$. Now as the extensions $\Zb \to \Yb$ 
and $\Xb \to \Zb$ are both relatively weakly mixing extensions, so is
the iterated extension $\Xb \to \Yb$. Finally, as distal extensions do not raise entropy,
the entropy of $\Yb$ is $< \del$.

(b)\ 
By definition there is a nontrivial factor map $\sig : \Xb \to \Zb$ such that the extension
$\sig$ is relatively weakly mixing. By the weak Pinsker property \cite{Au},
we have a decomposition $\Zb =  \Bb \times \Yb$ with $\Bb$ a Bernoulli system
and $\Yb$ having entropy $< \del$. Now clearly the composed map $\Xb \to \Yb$
satisfies the assertion of the proposition.
\end{proof}

\br

We denote by $\Mb$ the measure algebra associated to
$(\Xcal, \mu)$ consisting of equivalence classes of the relation $A \sim B \iff \mu(A \tri B) =0$,
equipped with the complete metric $d_\Mb(A, B) = \mu(A \tri B)$.

\br

In \cite{A-00} Ageev shows that
the generic automorphism of a Lebesgue space is conjugate to a $G$-extension 
for any finite abelian group $G$.
Thus, in particular it follows that the generic automorphism is not prime.
In \cite{A-03} he shows that the collection of $2$-fold quasi-simple systems
(and a fortiori that of simple systems) forms a meager subset of MPT or MPT$_e$,
the spaces of measure preserving  and ergodic measure preserving transformations respectively.

\br

\begin{prop}\label{RWMana}
RWM is an analytic subset of MPT.
\end{prop}

\begin{proof}
We consider the space MPT consisting of the invertible measure preserving transformations
of a nonatomic standard probability space $(X, \Xcal, \mu)$.
Given an ergodic $T$ in MPT, each positive set $A \in \Xcal, \ 0 < \mu(A) < 1$, determines a partition 
$\al = \{A, X \setminus A\}$ which, in turn, defines a $T$-invariant $\sig$-algebra
$\Acal = \bigvee_{N \in \N} \bigvee_{j = -N}^N T^j \al$.
We let $\pi : \Xb \to \Yb = (Y, \Ycal, \nu,T)$ be the corresponding factor map, so that 
$\Acal =  \pi^{-1}(\Ycal)$. 
By Proposition \ref{RWMdel} every RWM system $\Xb$ admits a nontrivial factor 
$\pi : \Xb \to \Yb$ with $\pi$ relatively weakly mixing and $\Yb$ with a two-set generator as above. 

Let $\mu = \int_Y \mu_y\, d\nu(y)$ be the disintegration of $\mu$ over $\nu$, and let 
$$
\la = \int_Y (\mu_y \times \mu_y)\, d\nu(y),
$$
be the relative product measure of $\mu$ with itself over $\nu$.
Using a Rohklin skew product representation for the ergodic system $\Xb$ as
$(X,\mu) = (Y \times Z, \nu \times \eta)$, we have $\mu_y = \eta$ for $\nu$  a.e. $y$,
and for functions $f$ and $g$ in $L^\infty(\mu)$ , we have
$$
\int_{X \times X} f(x_1)g(x_2) \, d\la(x_1,x_2) =
\int_Y \left(\int_Z f(y,z_1)\, d\mu_y(z_1) \cdot   \int_Z g(y,z_2)\, d\mu_y(z_2) \right) \, d\nu(y).
$$ 
Another way of writing $\int_Z f(y,z) \, d\mu_y(z)$, is $\E(f | \Acal)$, a function on $X$ measurable
with respect to $\Acal$. Therefore
\begin{equation}\label{ce}
\int_{X \times X} f(x_1)g(x_2) \, d\la(x_1,x_2) = \int_X \E(f | \Acal)(x) \cdot \E(g | \Acal)(x) \, d\mu(x).
\end{equation}
If 
$$
\Pcal_1 \prec \cdots \prec \Pcal_n \prec \Pcal_{n+1} \prec \cdots
$$
is a sequence of finite partitions such that the corresponding algebras $\hat{\Pcal}_n$ satisfy
$\bigvee_{n \in \N} \hat{\Pcal}_n = \Acal$ then, by the martingale convergence theorem,
\begin{equation}\label{mar}
\E(f | \Pcal_n) \to \E(f | \Acal).
\end{equation}
Now by definition the extension $\pi : \Xb \to \Yb$ is a relatively weakly mixing extension 
when the measure $\la$ is
ergodic. This is the case if, for a dense sequence of pairs of sets 
$\{(C_n, D_n)\}_{n \in \N}$ in $\Mb \times \Mb$, 
we have for all $n \in \N$
$$
\frac1L \sum_{i=0}^{L-1} (T \times T)^i (\ch_{C_n} \times \ch_{D_n}) 
\overset{L_2}\longrightarrow \int (\ch_{C_n} \times \ch_{D_n}) \, d\la,
$$
i.e. with $a_n =  \int (\ch_{C_n} \times \ch_{D_n}) \, d\la$,
$$
\int \left| \frac1L \sum_{i=0}^{L-1} (T \times T)^i (\ch_{C_n} \times \ch_{D_n}) - a_n \right|^2 d\la \to 0.
$$

Now expanding the expression
$$
 \int \left| \frac1L \sum_{i=0}^{L-1} (T \times T)^i (\ch_{C_n} \times \ch_{D_n}) - a_n \right|^2 \, d\la, 
 $$
 writing $\ch_{C_n} = f_n$ and $\ch_{D_n} = g_n$,
 and using (\ref{ce}), we get
 \begin{align*}
& \int \left| \frac1L \sum_{i=0}^{L-1} (T \times T)^i (\ch_{C_n} \times \ch_{D_n}) - a_n \right|^2 \, d\la\\
= &  \int \left| \frac1L \sum_{i=0}^{L-1} (T \times T)^i (f_n \times g_n) - a_n \right|^2 \, d\la\\
=& \frac{1}{L_2}  \sum_{i,j}  
\int  \Big( (T \times T)^i (f_n \times g_n)(T \times T)^j (f_n \times g_n) \\
&  -a_n (T \times T)^i (f_n \times g_n)   - a_n (T \times T)^j (f_n \times g_n)   + a_n^2 \Big) \, d\la \\
=& \frac{1}{L_2}  \sum_{i,j}   \int (T^i f_n \cdot T^j f_n) \times (T^i g_n \cdot T^j g_n) \, d\la - a_n^2 \\
=&  \frac{1}{L_2}  \sum_{i,j} \int
\E(T^i f_n \cdot T^j f_n | \Acal) \cdot \E(T^i g_n \cdot T^j g_n | \Acal) \, d \mu - a_n^2.
\end{align*}

Taking $\Pcal_M = \bigvee_{j = -M}^M T^j \al$ and  applying (\ref{mar}) we can approximate 
this  by the corresponding sum
$$
\frac{1}{L_2}  \sum_{i,j}^{L-1} \int
\left[\E(T^i f_n \cdot T^j f_n | \bigvee_{j = -M}^M T^j \al) \cdot 
\E(T^i g_n \cdot T^j g_n | \bigvee_{j = -M}^M T^j \al)\right] \, d \mu - a_n^2,
$$
which we denote by
$$
EA(C_n, D_n, L, M).
$$

Let $\{E_m\}_{m =1}^\infty$ be a sequence of sets in $\Xcal$ which is
dense in the subspace of $\Mb$ comprising sets $E$ with $\mu(E) > 1/10$.
Let $\{(C_n, D_n)\}_{n=1}^\infty$ be a dense sequence in $\Mb \times \Mb$.
For positive integers, $m, n, k, L, M, N$, we consider the set
$$
U(m, n, k, L, N, M) \subset \Mb \times {\rm MPT}_e,
$$
comprising those pairs $(A,T)$ of $\Mb \times \Aut_e(\mu)$ that,
with $\al = \{A, X \setminus A\}$ \ $0 < \mu(A) < 1$, satisfy the following inequalities:
\begin{enumerate}
\item
$d_{\Mb}(E_m, \bigvee_{j = -N}^N T^j \al) > 1/100$.
\item
$EA(C_n, D_n, L, M) < 1/k$,
\end{enumerate}
where the distance $d_{\Mb}(E_m, \bigvee_{j = -N}^N T^j \al)$ is defined as the minimum of the distances $d(E_m,B)$,
when $B$ ranges over the elements of the finite algebra generated by the partition 
$\bigvee_{j = -N}^N T^j \al$.

The set $U(m, n, k, L, N, M)$ is open and we let
$$
{\rm {rwm}} = 
\bigcup_{m=1}^\infty \bigcap_{N =1}^\infty \bigcap_{n =1}^\infty \bigcap_{k =1}^\infty \bigcup_{L=1}^\infty \bigcup_{M_0=1}^\infty \bigcap_{M = M_0}^\infty
U(m, n, k, L, N, M).
$$
Now, in view of Proposition \ref{RWMdel}, RWM is the projection of the set ${\rm{rwm}}$ in $\Aut_e(\mu)$; i.e.
$$
{\rm{RWM}} = \{T  \in \Aut(\mu) : \exists A \in \Mb, \ (A,T) \in {\rm{rwm}}\}.
$$  
As we have shown that the set rwm is Borel, it follows that the set RWM is analytic.
\end{proof}

\br

The next proposition is proved similarly and we use the notation of the previous proof.

\begin{prop}
TRWM is a co-analytic subset of MPT.
\end{prop}

\begin{proof}
By Claim \ref{TRWM0} it suffices to show that TRWM is a co-analytic subset of
the $G_\del$ set of $0$-entropy ergodic transformations.
We also recall that by Krieger's theorem every such transformation admits a two-set generator.

Let $\{(C_n, D_n)\}_{n=1}^\infty$ be a dense sequence in $\Mb \times \Mb$.
For positive integers, $n, k, L, M$ we consider the set
$$
 U(n, k, L, M) \subset \Mb \times \Aut_e(\mu),
$$
comprising those pairs $(A,T) \in \Mb \times \Aut_e(\mu)$ that,
with $\al = \{A, X \setminus A\}, \ 0 < \mu(A) <1$, satisfy the following inequality:
$$
EA(C_n, D_n, L, M) < 1/k,
$$

The set $ U(n, k, L, M)$ is open and we let
$$
{\rm {trwm}} = 
\bigcap_{n =1}^\infty \bigcap_{k =1}^\infty \bigcup_{L=1}^\infty \bigcup_{M_0=1}^\infty \bigcap_{M = M_0}^\infty
U(n, k, L, M).
$$
Now TRWM is the collection of $T \in \Aut_e(\mu)$ such that $T$ has zero entropy and
$(A,T) \in {\rm {trwm}}$ for all $A \in \Mb$,
i.e. 
$$
\Aut_e(\mu) \setminus {\rm {TRWM}} =
\{T \in \Aut(\mu) : \exists A \in \Mb, \ (A,T) \not \in {\rm {trwm}}\}.
$$
\end{proof}

\br

 \begin{prob}
Ryzhikov's question in \cite{R-06}, is whether the set RWM is co-meager.
 \end{prob}

\begin{rem}
In view of Definition \ref{def} (4) and Claim \ref{c1} we have that 
$$
{\rm MPT}_e= {\rm RWM} \sqcup {\rm TRD} = {\rm RWM} \sqcup {\rm TRD}^\sim \sqcup {\rm{PRIME}}
$$
and the last union is disjoint (note that the prime systems of the form $\Z_p$ do not belong
to $\Aut(\mu)$). 
Since by \cite{A-00} 
the collection PRIME (which is shown to be co-analytic in \cite{G-K}) is meager, we
conclude, by the zero-one law (see  \cite{G-K}), 
that one of the, clearly invariant, collections 
RWM and TRD (or RWM and TRD$^\sim$) is meager and the other comeager.
In the next section we will show that the set RWM is the comeager one.
\end{rem}

 In the next proposition we determine the, relatively low, complexity of the class PROP consisting of
elements of ${\rm MPT}_e $ which admit a proper factor.
 
 \begin{prop}\label{proper}
 \begin{enumerate}
 \item
 The set 
$$
{\rm{prop}} = \{(A, T) : \al = \{A , X \setminus A\},\ \Ycal = \bigvee_{N \in \N} \bigvee_{j = -N}^N T^j \al \subsetneq \Xcal\}
$$
is a $G_{\del, \sig}$ subset of $\Mb \times {\rm MPT}_e $
which is invariant under the diagonal action of $G = \Aut(\mu)$
on $\Mb \times {\rm MPT}_e $ defined by 
$$
g \cdot (A,T) = (gA, gTg^{-1}),   \quad A \in \Mb, \ T \in {\rm MPT}_e .
$$
\item
Its projection 
$$
{\rm{PROP}} = 
\{T \in {\rm MPT}_e : \exists \al,\ \Ycal = \bigvee_{N \in \N} \bigvee_{j = -N}^N T^j \al \subsetneq \Xcal\}
$$ 
is a nonempty $G_{\del, \sig}$ subset of ${\rm MPT}_e $.
 \end{enumerate}\end{prop}

\begin{proof}
(1) \ 
Let $\{E_m\}_{m =1}^\infty$ be sequence of sets in $\Xcal$ which is
dense in the subspace of $\Mb$ comprising sets $E$ with $\mu(E) \geq 1/10$.

For positive integers, $m, N$ set
$$
U(m, N)  
= \{(A,T) \in \Mb \times {\rm MPT}_e  : 
d_{\Mb}(E_m, \bigvee_{j = -N}^N T^j \al) > 1/100\},
$$
with $\al = \al(A) =\{A, X \setminus A\}$ (for $A \in \Xcal$ of positive measure),
The $U(m,N)$ is an open set and the set
\begin{align*}
U(m) &= \bigcap_{N = 1}^\infty U(m, N) \\
&= \{(A,T) \in \Mb \times {\rm MPT}_e  : 
d_{\Mb}(E_m, \bigvee_{N \in \N} \bigvee_{j = -N}^N T^j \al) > 1/100\}\\
\end{align*}
is a $G_\del$ set.
Finally we have
$$
{\rm prop}  = \bigcup_{m=1}^\infty U(m) = 
\{(A,T) \in \Mb \times {\rm MPT}_e ) : \bigvee_{N \in \N} \bigvee_{j = -N}^N T^j \al \not = \Xcal\},
$$
is a $G_{\del, \sig}$ set.
The invariance is clear.

(2) \ 
The projection $P : \Mb \times {\rm MPT}_e  \to {\rm MPT}_e $ is an open 
homomorphism of Polish dynamical systems.
Therefore each $P(U(m,N))$ is an open set, so that 
the image ${\rm{PROP}} = P({\rm {prop}})$ is a non-empty invariant  $G_{\del,\sig}$
subset of ${\rm MPT}_e $. 
\end{proof}

\br

\section{RWM is generic}\label{RWMg}

Let $(X, \Xcal, \mu)$ be a standard Lebesgue space, were the probability measure $\mu$ has no atoms. 
Let $\Bb$ denote the $W^*$-algebra of bounded linear operators on the Hilbert space
$L_2(\mu)$, equipped with the strong operators topology. It is easy to see that the space MPT,
comprising the Koopman operators corresponding to the invertible measure preserving transformations
of $(X, \Xcal,\mu)$ is a closed subset of $\Bb$.
Let $\Pb$ be the set of positive projections, i.e.
those elements $P$ of $\Bb$ such that $\|P\|=1, P^2 =P$, 
$P(\ch) =1$, and $Pf \geq 0$ for every $0 \leq f \in L_2(\mu)$.
This is a closed subset of $\Bb$.

A sub-$\sig$-algebra $\Ycal \subset \Xcal$ determines 
a standard probability space $(Y,\Ycal, \nu)$, a closed subspace $H = H(\Ycal) \subset L_2(\mu)$,
and a positive projection $P : L_2(\mu) \to H$.
More precisely, a positive projection $P$ can be identified with a {\em conditional expectation
operator} over the subspace $H = PL_2(\mu)$ that can be described as
the set of functions in $L_2(\mu)$ which are measurable with respect
to a sub-$\sig$-algebra $\Ycal \subset \Xcal$. 
As was shown by Rokhlin, we can realize the latter inclusion as a map
$(X, \Xcal, \mu) \to (Y, \Ycal, \nu)$, where $(Y, \Ycal,\nu)$ is a standard Lebesgue
space, so that the projection $P$ has the form
$$
(Pf)(y) = \int_Y f(x) \, d\mu_y(x), \quad {\text{for}}\ \nu \ {\rm{a.e.}}\  y \in Y.
$$
Clearly the collection of projections $P$ whose range has dimension $\leq k$ for $k \in \N$
is a closed set and it follows that the collection $\Pb_i$ of projection
whose range is infinite dimensional forms a $G_\del$ subset of $\Pb$.
We denote by $\Pb_{ci}$ the set of $P \in \Pb$ such that, 
$P \not = \Id$, and such that in the corresponding Rokhlin representation,
in the disintegration 
$\mu = \int_Y \mu_y\,d\nu(y)$, a.e. $\mu_y$ has no atoms.
Finally let 
$\Qb = \Pb_i \cap \Pb_{ci}$.

We denote by $\Tb$
the collection of weakly mixing MPT's,  which we often identify with their
Koopman operators. By Halmos' theorem $\Tb$ is a dense $G_\del$ subset of MPT.
Thus we view the elements of both $\Pb$ and $\Tb$ as operators in $\Bb$.
Note that if $T$ is weakly mixing and $(X, \Xcal,\mu,T) \to (Y, \Ycal,\nu, T \rest \Ycal)$ is 
a nontrivial factor, then the corresponding $P$ is necessarily in $\Pb_i$.
The group $G = \Aut(\mu)$ acts continuously on $\Bb$ by conjugations,
and both $\Tb$ and $\Qb$ are invariant subsets under this action.
Let 
$$
\Lb = \{ (T,P) : PT =TP,  \  P \in \Qb,  \  T \in \Tb\}.
$$
and we will consider the diagonal action of $\Aut(\mu)$ on $\Lb$.
Denote by $\pi_1$ and $\pi_2$ the projections of $\Lb$ on its first and second coordinates, respectively.
Note that if $T = \pi_1(T,P)$  for $(T,P) \in \Lb$, then $T$ is weakly mixing and is not prime, and
if $P = \pi_2(T,P)$  for $(T,P) \in \Lb$ then the operator $TP$, defined on the infinite dimensional
space $PL_2(\mu)$, is a factor of $T$ and is therefore weakly mixing on $L_2(Y, \Ycal, \nu)$
with $\nu$ atomless. 

\begin{lem}
A pair $(T,P)$ is in $\Lb$ iff $T$, as a transformation in MPT, is weakly mixing and it admits 
a factor map $\Xb =(X, \Xcal, \mu,T) \to \Yb = (Y, \Ycal, \nu, S)$
with $PL_2(\mu) \cong L_2(\nu) = L_2(\Ycal)$, where $\nu$ has no atoms, and 
$\Xb$ has the form of a Rokhlin skew product  
$$
(X, \Xcal, \mu, T) = (Y \times Z, \Ycal \otimes \Zcal, \nu \times \eta, \hat{S}),
$$
where $(Z, \Zcal, \eta,T)$ is a standard Lebesgue space with $\eta$ atomless, and $S : Y \to \Aut(Z, \eta)$
is a measurable cocycle with $\hat{S}(y,z) = (Ty,  S_y z)$ \ $\mu$-a.e. 
\end{lem}

\begin{proof}
If $(T,P) \in \Lb$ then $T$ is weakly mixing, hence ergodic and it then follows that the Rokhlin
presentation has this form (see e.g. \cite[Theorem 3.18]{G}).
Conversely, clearly if $T$ is a weakly mixing skew product as described then the corresponding 
pair $(T,P)$ is in $\Lb$.
\end{proof}

\begin{prop}
\begin{enumerate}
\item
$\Qb$ is a $G_\del$ subset of $\Bb$.
\item
$G$ acts transitively on $\Qb$.
\item
$\Tb$ is a $G_\del$ subset of $\Bb$
\item
The action of $G$ on the Polish space $\Tb$ is minimal.
\item
$\Lb$ is a closed subset of the Polish space $\Tb \times \Qb$.
\item
$G$ acts minimally on $\Lb$.
\end{enumerate}
\end{prop}

\begin{proof}
(1)\ 
We already observed that the condition (i) that $PL_2(\mu)$ be infinite dimensional defines a 
$G_\del$ subset, $\Pb_i$ of $\Bb$.
Thus a positive projection $P \in \Pb_i$ is in $\Qb$ iff it is in $\Pb_{ci}$.
We claim that this is equivalent to the following condition:

For all positive measure sets $A \in \Xcal$ there exist sets $B$ and $C$ such that 
\begin{enumerate}
\item[(i)]
$\ch_A = \ch_B + \ch_C$,
\item[(ii)]
$P(\ch_B) = P(\ch_C)$
\end{enumerate}

Now we can express this property as the intersection of a countable collection
of open sets as follows.
Let $\{A_n\}$ be a dense sequence in the measure algebra and define:
\begin{gather*}
U(N , n , i, j) = \\
\{ P \in \Pb : \mu( A_n  \tri ( A_i \cup A_j)) < 1/N,  \  \| P(\ch_{A_i}) - P(\ch_{A_j}) \| < 1/N,
\ \&  \ \mu(A_i \cap A_j) < 1/N \}.
\end{gather*}
Now set
$$
\Pb_{ci} = \bigcap_{ n \ge 1} \bigcap_{N \geq1} \bigcup_{(i, j)} U(N , n , i, j).
$$
To see that this intersection captures the property that the conditional measures are continuous we
argue by contradiction.
If $\{\mu_y\}_{y \in Y}$ are the conditional measures  associated to the projection $P$,
and for a set $B \subset Y$ with $\nu(B)>0$, there are atoms of $\mu_y$ for all $y \in B$ with
measure $\ge c >0$, then there is a measurable function $f : B \to X$ such that $\mu(\{f(y)\}) \geq c,
\ y \in B$.
The range $A = f(B)$ is a measurable subset of $X$ and
$$
(P\ch_A)(y) = \mu_y(\{f(y)\}).
$$
Now we take $1/N \ll \mu(A)$ and $A_{n_0}$ such that $\mu(A \tri A_{n_0}) < 1/N$.
We see that if 
$$
\mu(A \tri (A_{i_0} \cup A_{j_0})) < 1/N
$$
then the approximate equality $P\ch_{A_{i_0}} \approx P\ch{A_{j_0}}$ cannot hold.
Thus $P$ is not in the intersection $\Pb_{ci}$, as required.

To prove Claim (2) observe that $P$ determines a Rokhlin decomposition 
of $\Xcal$ over $\Ycal$ as explained above, and that clearly any two such
decompositions are measurably isomorphic.
Claim (3) is well known.
Claim (4) follows from Halmos' Conjugacy Lemma \cite[page 77]{H},
which asserts that the conjugacy class of each aperiodic 
element of $\Aut(\mu)$ is dense.
Claim (5) is easy to check. Finally claim (6) can be easily deduced from claim (2) and 
 \cite[Proposition 2.3]{G-W}, which is a relative analogue of Halmos' conjugacy lemma.
\end{proof}

\begin{rmk}\label{A-SE}
In the definition of $\Qb$ we excluded the two extreme cases, when $H = L_2(\mu)$
(i.e. when $P = \Id$) and when $H $ is finite dimensional.
We note however that due to this exclusion, in the projection of $\Lb$ on the $\Tb$ coordinate,
the weakly mixing transformations that are missing are:
(i)  The prime transformations, which
according to Ageev \cite{A-00} form a meager subset of $\Tb$.
(ii) Those weakly mixing transformation for which every factor map $\Xb  \to \Yb$ is finite to one.
Such transformations are in TRD, and e.g. by Ageev \cite{A-03a}, or by
Stepin and Eremenko ---
who show in \cite{S-E} that the generic ergodic transformation
admits the infinite torus in its centralizer, form a meager set.
\end{rmk}

Recall the following (see  \cite[Apendix A]{M-T} and \cite[Definition 2.7, Definition 2.4]{M}).

\begin{defn}
 Let $Y, Z$ be Polish spaces and let $f : Y \to Z$ be a continuous map. 
 \begin{enumerate}
 \item
 We say that $f$  is {\em category-preserving} if it satisfies the following condition:
For any comeager $A \subseteq Z, f^{-1}(A)$ is comeager in $Y$.
\item
A point $y \in Y$ is a point of {\em local density for $f$} if
for any neighborhood $U$ of $y$ the set $\ol{f(U)}$ is a neighborhood of $f(y)$.
 \end{enumerate}
\end{defn}

We then have the following (\cite[Proposition 2.8]{M}, see also Tikhonov's work \cite{T}):

\begin{prop}
Let $Y, Z$  be Polish spaces and $f : Y \to Z$ a continuous map. 
Then $f$ is category-preserving if and only if the set of points which are locally dense for $f$
is dense in $Y$.
\end{prop}

We also recall the following well known observation called  
ÒDougherty's lemmaÓ (see e.g. \cite{K-00} and \cite[Proposition 2.5]{M}).

\begin{prop}\label{Dou}
Let  $Y, Z$  be Polish spaces, $f : Y \to Z$ a continuous map
such that the set of points which are locally dense for $f$ is dense in $Y$.
Let $B \subset Y$ be a comeager subset of $Y$. Then $f(B)$ is not meager in $Z$.
\end{prop}

The following statement is shown in \cite{M-T}.

\begin{thm}\label{MT}
Let $Y, Z$ be Polish spaces, and $f : Y \to Z$ be a category-preserving map. 
Let also $A$ be a subset of $Y$ with the property of Baire. Then the following assertions are equivalent:
\begin{enumerate}
\item[(i)]  
$A$ is comeager in $Y$.
\item[(ii)] 
$\{z \in Z : A \cap  f^{-1}(z) \ {\text{ is comeager in}} \  f^{-1}(z)\} $ is comeager in $Z$.
\end{enumerate}
\end{thm}

Next we have a dynamical version of the Kuratowski-Ulam theorem,  \cite[Proposition 2.10]{M}:

\begin{prop}\label{min}
Let $H$ be a Polish group, $Y, Z$ be two Polish $H$-spaces and 
$f : Y \to Z$ a $H$-map. Assume that $Y$ is minimal (i.e. every orbit is dense) and 
$f (Y )$ is not meager. Then $f$ is category-preserving.
\end{prop}

\begin{rmk}
An older result of Veech \cite[Proposition 3.1]{V-70} shows this  result for compact topological systems.
See also \cite[Lemma 5.2]{G-90} where a topological analogue of simplicity was studied.
\end{rmk}

The next theorem answers positively Ryzhikov's question.

\begin{thm}\label{thm-RWM}
The property RWM is generic.
\end{thm}

\begin{proof}
We first show that RWM is comeager.
Consider the closed set $\Lb \subset \Tb \times \Qb$.
By \cite{S} (see also \cite[Theorem 1.3]{G-W}) we know that for every fixed $P \in \Qb$ the set
of elements $T \in \Tb$ such that $T$ is a relatively weakly mixing extension of
$T \rest \Ycal$, where $\Ycal$ is the $\sig$-algebra determined by $P$,
is a dense $G_\del$ subset of $\Lb^P$,
where
$$
\Lb^P = \{ T \in \Tb : (T,P) \in \Lb\}.
$$
Consider the set
$$
\Lb_{\rm{RWM}} = \{(T,P) \in \Lb : T \ {\text{is relatively weakly mixing over the factor determined by}} \ P\}. 
$$
By Proposition \ref{RWMana} the set RWM$\ \cap \ \Tb$ is an analytic subset of $\Tb$, hence has the BP.
Denoting by $\pi_1$ the projection from $\Tb \times \Qb$ onto $\Tb$ we observe that
$\Lb_{\rm{RWM}} = \Lb \cap \pi_1^{-1}(RWM)$ and as $\Lb$ is a closed subset of $\Tb \times \Qb$,
we conclude that $\Lb_{\rm{RWM}}$ has the BP.

Applying Proposition \ref{min} to the minimal Polish system 
$(\Lb,G)$ and the projection map $\pi_2 : (\Lb,G) \to  (\Qb,G)$,
we conclude that the map $\pi_2$ is category preserving.
(Note that $\pi_2(L) = \Qb$.)  
Applying Theorem \ref{MT} to $\Lb_{\rm{RWM}} \subset \Lb$,
we conclude that $\Lb_{\rm{RWM}} $ is comeager in $\Lb$.

Now by the same token also the other projection map 
$\pi_1 : (\Lb,G) \to  (\Tb,G)$ is category preserving and so, 
by Dougherty's lemma, Proposition \ref{Dou}, 
the image of the set $\Lb_{\rm{RWM}}$ under $\pi_1$, namely
the subset $\pi_1(\Lb_{\rm{RWM}})\subset \Tb$, is not meager in $\pi_1(\Lb) \subset \Tb$.
Since this subset is also invariant, it follows from the zero-one law 
that it is comeager in $\Tb$.
We now note that this subset is 
exactly 
the set of nonprime, weakly mixing systems which are RWM
and, in view of Remark \ref{A-SE}, it follows that the set
RWM \ $\cap  \ \Tb$ is indeed comeager in $\pi_1(\Lb) \subset \Tb$.
\end{proof}

\begin{rmk}
Our proof of Theorem \ref{thm-RWM} uses Ageev's result, which asserts that the class PRIME
is meager, so of course it does not provide a new proof of this statement.
However note that whereas Ageev's proof demonstrates the prevalence of group extensions, 
ours demonstrates the prevalence of relatively weakly mixing extensions.
\end{rmk}

\br

\section{The relatively weakly mixing factors of a positive entropy 
$T$ form a dense $G_\del$ set}\label{thouvenot}

In the second part of the paper, Sections 3 to 7, we change slightly the framework of our discussion.
We mainly fix an ergodic transformation $T$ with a certain property,
and then consider either its family of factors, 
as in the present section, or its family of ergodic extensions, as in Sections 4 to 7.

Recall that $\Pb$ denotes the $G_\del$ set of positive projections 
in $\Bb$, the $W^*$-algebra of bounded linear operators on $L_2(\mu)$.

\begin{prop}\label{Gdelta}
For a fixed $T \in {\rm MPT}_e$, the set $\Wb_T$ of those $P \in \Pb$ such that 
$TP = PT$, and such that $T$ is relatively weakly mixing over the factor determined by $P$, 
is a $G_\delta$ set.
\end{prop}

\begin{proof}
Let $\Pb_T$ denote the closed subset of $\Pb$ comprising the projections $P$
satisfying $TP = PT$.
First note that for two bounded functions functions $f, g$  in $L_2(\mu)$
the integral of the product  $f(x_1)g(x_2)$ with respect to the relative product
measure $\la  =\mu\underset{\nu}\times \mu$ 
over  the factor $\Yb$ defined by $P \in \Pb_T$ is given by:
$$
 \int f(x_1)g(x_2) \, d\la(x_1, x_2)= \int  Pf(x)Pg(x) \,d\mu(x).
$$ 
Now we fix  a sequence $\{f_i\}$ of bounded functions that are dense in
$L_2(\mu)$ and then define the sets
$$
U(k, N, i, j ) =  \left\{P \in \Pb_T:  
\int  \left |  1/N  \sum_{n=0}^{N-1} T^nPf_i(x) T^nPf_j(x)  - a_{ij} \right|^2 \,d\mu(x)  < 1/k\right\}
$$
where $a_{ij} =  \int  Pf_i(x)Pf_j(x) d\mu(x)$.
These are open sets and we let
$$
\Wb =  \bigcap_{(i, j)} \bigcap_{ k \ge 1} \bigcup_{N \geq1}  U(N , n , i, j).
$$
Now $\Wb$ is a $G_\del$ set and if $P \in \Wb$ then $T$ is RWM. 
Clearly also when $P \in \Pb $ is such that 
$PT = TP$ and
the extension over the factor determined
by $P$ is weakly mixing, then $P \in \Wb$. Thus $\Wb = \Wb_T$. 
\end{proof}

\br

\begin{thm}
 If $T \in {\rm MPT}_e$ has positive and finite entropy then among the positive projections
$P$ which commute with $T$ (i.e. among the factors of $T$), 
for a dense $G_{\delta}$ set the corresponding extension over
the factor determined by $P$ is relatively weakly mixing. 
\end{thm}

\begin{proof}
This is just the set $\Wb_T$ above and hence, by Proposition \ref{Gdelta}, it is a $G_\delta$ set.

To see that it is dense we use the fact that is proved by Thouvenot for a $T$ with the weak-Pinsker property
\cite[Lemma 7]{Th}. 
Namely, given $\ep >0$
one can perturb by no more than $\ep$ any partition $\Qcal$ to a partition $\bar{Q}$
so that the the new partition splits off with a Bernoulli complement $\bar{\Bcal}$:
\begin{enumerate}
\item
$|\bar{\Qcal} - \Qcal | < \ep$,
\item
$(\bar{\Qcal})_T \perp (\bar{\Bcal})_T$,
\item
$(\bar{\Qcal})_T \vee (\bar{\Bcal})_T = \Xcal$,
\item
The partitions $T^i \bar{\Bcal}, \ i \in \Z$, are independent.
\end{enumerate}

Now given a projection $P \in  \Pb_T$ that one wants to approximate, we find a generator $\Qcal$ 
for the factor determined by $P$ and apply Thouvenot's lemma.

However, to make sure that the Bernoulli part of Thouvenot's splitting is not trivial we first modify 
the generator $\Qcal$ by a little bit so that the entropy of the new $\Qcal$
is strictly less than the entropy of $T$.
Then we choose the $\epsilon$ in the conclusion of \cite[Lemma 7]{Th} so small that
the factor defined by $\bar{\Qcal}$ still
has less than full entropy. Now we are sure that the Bernoulli complement is non-trivial. 
Finally, by Austin's recent result \cite{Au}, the weak Pinsker property always holds and our proof is complete.
\end{proof}

\begin{prob}
Can one prove an analogous claim for relative Bernoulli extensions ?
\end{prob}

\br

\section{A generic extension does not add entropy }\label{add-0}

In the last few sections of the paper we show that for a fixed ergodic $T$ with property $\Ab$, 
a generic extension $\hat{T}$ of $T$ also has the property $\Ab$.
Here $\Ab$ stands for each of the following properties:
(i) having the same entropy as $T$, (ii) Bernoulli, (iii)  K, and (iv) loosely Bernoulli.

\begin{thm}\label{+0}
For any fixed ergodic transformation $T $, the generic extension does not increase entropy.
\end{thm}

\begin{proof}
Let $\Xb = (X, \mathcal{X},\mu,T)$ be an ergodic system with finite entropy,
which for convenience we assume it equals $1$. Let $\Rcal \subset \Xcal$ be a finite 
generating partition with entropy $1$.
Let $\Scal$ be the collection of Rokhlin cocycles with values in MPT$(I, \la)$, where 
$\la$ is the normalized Lebesgue measure on the unit interval $I =[0,1]$.
Thus an element $S \in \Scal$ is a measurable map $x \mapsto S_x \in $ MPT$(I, \la)$,
and we associate to it the {\em skew product transformation}
$$
\hat{S}(x,u) = (Tx, S_x u),\quad  (x \in X, u \in I).
$$ 
Let $Y = X \times I$ and set $\Yb = (Y, \mathcal{Y}, \mu \times \la)$, with $\Ycal = \Xcal \otimes \Bcal$.

We recall that, by Rokhlin's theorem,  every ergodic extension $\Yb \to \Xb$ either has this form
or it is $n$ to $1$ a.e for some $n \in \N$ (see e.g. \cite[Theorem 3.18]{G}).
Thus the collection $\Scal$ parametrises the ergodic extensions of $\Xb$ with infinite fibers.
This defines a Polish topology on $\Scal$ which is inherited from MPT$(X \times I, \mu \times \la)$.
Of course a finite to one extension does not add entropy and thus it suffices to show that 
for a dense $G_\del$ subset $\Scal_0 \subset \Scal$ we have
$h(\hat{S}) = 1$ for every $S \in \Scal_0$.

For each $n \in \N$ let $\Qcal_n$ denote the dyadic partition of $[0,1]$
into intervals of size $1/2^n$, and let 
$$
\Pcal_n = \Rcal \times \Qcal_n.
$$
Clearly, for any $S \in \Scal$ we have $h(\Pcal_n,\hat{S}) \geq 1$.

\begin{lem}
For any $\ep >0$, the set
$$
U(n, \ep) =
\{S \in \Scal : h(\Pcal_n,\hat{S}) < 1 + \ep\},
$$
is open.
\end{lem}

\begin{proof}
Let $S_0 \in U(n, \ep)$. Then, there exists $N$ such that
$$
\frac{H(\vee_{0}^{N-1} S_0^{-i} \Pcal_n)}{N} = 1+a,
$$
with $0 \leq a < \ep$.
If $\del = \frac{\ep - a}{2}$ and $S$
is sufficienltly close to $S_0$, then
$$
\left| \frac{H(\vee_{0}^{N-1} S^{-i} \Pcal_n)}{N} - 
\frac{H(\vee_{0}^{N-1} S_0^{-i} \Pcal_n)}{N} \right| < \del,
$$
so that $S$ will be in $U(n,\ep)$.
\end{proof}

It follows from the lemma that
$$
\Scal_0 = \bigcap_{n \in \N} \bigcap_{k \in \N} U(n, 1/k),
$$
is a $G_\del$ subset of $\Scal$ consisting of those $S \in \Scal$ such that $h(\hat{S}) =1$.
Clearly $\Scal_0$ is nonempty; e.g. we can take $S$ to be the constant cocycles with a fixed
value $R_\al$, an irrational rotation on the circle.
Thus by the relative Halmos theorem  \cite[Proposition 2.3]{G-W}, it follows that $\Scal_0$ is a dense 
$G_\del$ subset of $\Scal$, as claimed. 
\end{proof}

\br

Using the classical Kuratowski-Ulam theorem we can ``lift" this theorem to the collection of ergodic 
transformations which commute with a fixed $P \in \Qb$  (see Section \ref{RWMg}).

As usual let $(X, \Xcal, \mu)$ be a standard atomless Lebesgue space and let
$\Ycal \subset \Xcal$ be a $\sig$-algebra such that $\nu = \mu \rest \Ycal$ is atomless
and such that 
$$
(X, \Xcal, \mu) = (Y \times Z, \Ycal \otimes \Zcal, \nu \times \eta),
$$
where $(Z, \Zcal, \eta)$ is a standard Lebesgue space with $\eta$ atomless.
Let $P \in \Pb$ be the  projection of $L_2(\mu)$ on $L_2(\Ycal)$.

Set
\begin{align*}
\mathcal{N} &= \{T \in \Aut(\mu) : T \ {\text{leaves the $\sig$-algebra $\mathcal{Y}$ invariant}}\}\\
&=  \{T \in \Aut(\mu) : TP = PT\},
\end{align*}
and let
$$
\Ncal_0  = \{T \in \mathcal{N} : T \ {\text{is ergodic and does not add entropy to}}\
T \rest \mathcal{Y}\}.
$$
As the collection of ergodic transformations is a dense $G_\del$ subset of $\Aut(\nu)$,
Theorem \ref{+0}, combined with the Kuratowski-Ulam theorem 
\cite[Theorem 15.4]{Ox-80}, immediately yield
the following result:

\begin{cor}
The collection $\Ncal_0$ forms a residual subset of $\mathcal{N}$.
\end{cor}

\br

\section{A generic extension of a Bernoulli system is Bernoulli }\label{bernoulli}

\begin{thm}
For any fixed Bernoulli transformation $T$ of finite entropy the generic extension is Bernoulli.
\end{thm}

\begin{proof}
We first recall the following:

\begin{defn}
A finite measurable partition $\Pcal$ of $X$ is called {\em very weak Bernoulli} (VWB for short) if
for every $\ep>0$, there is a positive integer $N$ such that for all $k \geq 1$, there is a
collection $\Gcal_k$ of atoms of the partition $\bigvee_{i = -k}^{-1} T^{-i}\Pcal$ such that
\begin{enumerate}
\item
$\mu(\bigcup \{A \in \Gcal_k\}) > 1 -\ep$.
\item
$\bar{d}_N\left(\vee_{i = 0}^{N -1} T^{-i}\Pcal, \vee_{i = 0}^{N -1} T^{-i}\Pcal \rest A \right) < \ep$,
for every atom $A$ of $\Gcal_k$.
\end{enumerate}
Here $\bar{d}_N$ is the normalized Hamming distance of distributions, 
and $\bigvee_{i = 0}^{N -1} T^{-i}\Pcal \rest A$ is the 
conditional distribution of $\bigvee_{i = 0}^{N -1} T^{-i}\Pcal$
restricted to the atom $A$ of the partition $\bigvee_{i = -k}^{-1} T^{-i}\Pcal$.
\end{defn}

As the inverse limit of Bernoulli systems is Bernoulli,
to show that a transformation $T$ on $(X, \Xcal, \mu)$ is 
Bernoulli it suffices to show that for a refining sequence of partitions
$$
\Pcal_1 \prec \cdots \prec \Pcal_n \prec \Pcal_{n+1} \prec \cdots
$$
such that the corresponding algebras $\hat{\Pcal}_n$ satisfy
$\bigvee_{n \in \N} \hat{\Pcal}_n = \Xcal$, for each $n$, the process $(T, \Pcal_n)$ is VWB.

With no loss of generality we assume that the entropy of the Bernoulli transformation $T$ is $1$.
We keep the notations of the previous section (Section \ref{add-0}).
and we will show that the generic element of the space $\Scal_0$ is Bernoulli.
By Theorem \ref{+0} for every $S \in \Scal$ the corresponding extension $\hat{S}$ of $T$ has 
entropy $1$.
We want to show that for a generic set of $S \in \Scal_0$, the corresponding $\hat{S}$ is Bernoulli.
To do this we need to express the fact that the partition $\Pcal_n$ is VWB in a ``finite way". 

The problem is that in the definition of the VWB property the inequality 
$$
\bar{d}_N\left(\vee_{i = 0}^{N -1} T^{-i}\Pcal_n, \vee_{i = 0}^{N -1} T^{-i}\Pcal_n \rest A \right) < \ep,
$$
for $A \in \Gcal_k$, has to hold for all large $k$.
The key to being able to express this with one value  of $k_0$ lies in the fact that we know that the entropy
of $\Pcal_n$ with respect to $\hat{S}$ for $S \in \Scal_0$ equals 1.
This is done as follows.

\begin{defn}
For two labeled partitions
$\Pcal = \{P_1,\dots, P_a\}$ ,
$\Qcal = \{Q_1, \dots, Q_b\}$ of $(X, \mu)$ we say that $\Pcal$ is $\ep$-independent of $\Qcal$
if there is a set of indices $J \subset \{1, \dots, b\}$ such that
\begin{enumerate}
\item
$\sum_{j \in J} \mu(Q_j) > 1 - \ep$. 
\item
$\sum_{i=1}^a | \mu(P_i | Q_j) - \mu(P_i) | < \ep$ for each $j \in J$.
\end{enumerate}
\end{defn}

One can use entropy to express independence since $H(\Pcal |\Qcal) = H(\Pcal)$ if and only if
$\Pcal$ is independent of $\Qcal$. If the number of elements of $\Pcal$ is fixed then the following is also 
well known.
Given $\ep >0$ there is a $\del > 0$ such that if
$H(\Pcal) < H(\Pcal |\Qcal) + \del$, then $\Pcal$ is $\ep$-independent of $\Qcal$.
The conditional version of this also holds. Namely,
if $\Rcal$ is a third partition, given $\ep >0$ there is a $\del > 0$ such that if
\begin{equation}\label{HPRQ}
H(\Pcal | \Rcal) < H(\Pcal |\Rcal \vee \Qcal) + \del,
\end{equation}
then $\Pcal$ conditioned on $\Rcal$ is $\ep$-independent of $\Qcal$.
That is, there is a set of atoms $A$ of $\Rcal$ whose union has total measure $> 1- \ep$, such that
conditioned on $A$, $\Pcal$ is $\ep$-inependent of $\Qcal$.

Here is another general fact. If the entropy of a partition $\Pcal$ in $(X, \Xcal, \mu, T)$ equals $h$,
then for all $N$,
\begin{equation}\label{Nh}
H( \vee_{i=0}^{N-1} T^{-i} \Pcal \, |  \vee_{i=- \infty}^{-1} T^{-i} \Pcal) =  Nh,
\end{equation}
hence
$$
Nh \leq H( \vee_{i=0}^{N-1} T^{-i} \Pcal \, |  \vee_{i=- k}^{-1} T^{-i} \Pcal) 
= H( \vee_{i=0}^{N-1} T^{-i} \Pcal \, |  (\vee_{i=- k_0}^{-1} T^{-i} \Pcal) \vee (\vee_{i=-k }^{-k_0 -1} T^{-i} \Pcal)),  
$$
for all $k  > k_0 \geq 1$.

Now in the inequality (\ref{HPRQ}) we let
$\Pcal = \vee_{i=0}^{N-1} T^{-i} \Pcal, \ \Rcal = \vee_{i=- k_0}^{-1} T^{-i} \Pcal$
and $\Qcal =   \vee_{i=-k }^{-k_0 -1} T^{-i} \Pcal$.
It follows that, given $\ep > 0$, there is a $\del >0$ such that, if $k_0$ is sufficiently large so that
$$
H( \vee_{i=0}^{N-1} T^{-i} \Pcal \, |  \vee_{i=- k_0}^{-1} T^{-i} \Pcal) < Nh +\del,
$$
then, by (\ref{Nh}),
$$
H( \vee_{i=0}^{N-1} T^{-i} \Pcal \, |  \vee_{i=- k_0}^{-1} T^{-i} \Pcal) < 
H( \vee_{i=0}^{N-1} T^{-i} \Pcal \, |  (\vee_{i=- k_0}^{-1} T^{-i} \Pcal) \vee (\vee_{i=-k }^{-k_0 -1} T^{-i} \Pcal)) +\del,
$$
hence, by (\ref{HPRQ}),
$\vee_{i=0}^{N-1} T^{-i} \Pcal$ conditioned on $ \vee_{i=- k_0}^{-1} T^{-i} \Pcal$,
is $\ep$-independent of $\vee_{i=-k }^{-k_0 -1} T^{-i} \Pcal$, for all $k > k_0$.

With this background we come to the main step of the proof.
Define the set $U(n, N_1, N_2, \ep, \del)$ to consist of those $S \in \Scal_0$ that satisfy:
\begin{enumerate}
\item
$H(\bigvee_{i=0}^{N_1 -1} \hat{S}^{-i} \Pcal_n \, | \bigvee_{i=-N_2}^{-1} \hat{S}^{-i} \Pcal_n ) < N_1 + \del$,
\item
$\bar{d}_{N_1} \left(\bigvee_{i=0}^{N_1 -1} \hat{S}^{-i} \Pcal_n,  
\bigvee_{i=0}^{N_1 -1} \hat{S}^{-i} \Pcal_n \rest A \right)  < \ep$, 
for a set $\Gcal_{N_2}$ of atoms $A \in  \bigvee_{i=-N_2}^{-1} \hat{S}^{-i} \Pcal_n $,
such that $(\mu \times \la)\left(\bigcup \{A : A \in \Gcal_{N_2}\} \right) > 1 - \ep$.
\end{enumerate}
We claim that the sets $U(n, N_1, N_2, \ep, \del)$ are open (easy to check) and that the $G_\del$ set
$$
\Scal_1 = \bigcap_{n, k, l} \bigcup_{N_1, N_2} U(n, N_1, N_2, 1/k, 1/l)
$$
comprises exactly the elements $S \in \Scal_0$ for which the corresponding $\hat{S}$ is Bernoulli. 
Thus, if $S \in \Scal_0$ is such that $\hat{S}$ is Bernoulli, then for every $n, \ep, \del$, there are
$N_1, N_2$ such that $S \in U(n, N_1, N_2, \ep, \del)$, and conversely, every for Bernoulli $\hat{S}$ 
the corresponding $S$ is in $\Scal_1$.

Note that (1) holds for any $N_1$ and $\del$ if only $N_2$ is sufficiently large, because
for $S \in \Scal_0$
$$
H(\vee_{i=0}^{N_1 -1} \hat{S}^{-i} \Pcal_n \, | \vee_{i=-\infty}^{-1} \hat{S}^{-i} \Pcal_n ) = N_1.
$$

Finally the collection $\Scal_1$ is nonempty; e.g, by a deep result of Rudolph \cite{Ru-79, Ru-85},
every weakly mixing group extension of $T$ is in $\Scal_1$. In fact an explicit example
of such an extension of the $2$-shift is given by Adler and Shields, \cite{A-S}.

Again we now apply the relative Halmos theorem  \cite[Proposition 2.3]{G-W}, to deduce that 
$\Scal_1$ is a dense $G_\del$ subset of $\Scal$, as claimed. 
\end{proof}

\br

\section{A generic extension of a K system is K}\label{K-Sec}

\begin{thm}\label{Kthm}
For any fixed K transformation $T$, the generic extension is K.
\end{thm}

\begin{proof}
%
%
The process $(T, \Pcal)$ is a K-process if it has a trivial tail, i.e.
$$
\bigcap_{k =1}^\infty \bigvee_{i= -\infty}^{-k} T^{-i}\Pcal
$$
is the trivial $\sig$-field. Thus this can be expressed by the property that  for all $N$ and $\ep >0$
there is a $k_0$  such that $\vee_{i=0}^{N-1} T^{-i}\Pcal$ is $\ep$-independent of
$\vee_{i=-\infty}^{-k} T^{-i}\Pcal$ for all $k \ge k_0$.
As we have seen one can express the $\ep$-independence in terms of entropy so that the
K-property can also be expressed by saying that for all $N$ and $\ep > 0$, there is a $k_0$ such that
for all $k  \geq k_0$
\begin{equation}\label{K}
H( \vee_{i=0}^{N-1} T^{-i} \Pcal \, |  \vee_{i= - k}^{-k_0} T^{-i} \Pcal)  >   H( \vee_{i=0}^{N-1} T^{-i} \Pcal) - \ep.
\end{equation}

To rid ourselves of the need for (\ref{K}) for all $k  \geq k_0$, notice that for all $n$ if $k_1$ is sufficiently large then
\begin{equation}\label{n}
H( \vee_{i=0}^{n-1} T^{-i} \Pcal \, |  \vee_{i= - k_1}^{-1} T^{-i} \Pcal)  <  nh +  \del = n + \del,
\end{equation}
where $h=1$ is the entropy of the process $(T, \Pcal)$. 
Now given $\ep$ if we take $k_0$ as above and for $n=N+k_0$ and $\del$ sufficiently small 
find $k_ 1$ so that (\ref{n}) holds then (\ref{K}) for $k=k_0 +k_1$ will
imply (\ref{K}) with $2\ep$ , instead of $\ep$ for all $k \geq k_1+k_1$.

Remark too that if a system is K for each member of a refining sequence of partitions then
it is K, as the inverse limit of K automorphisms is K.

So, 
keeping the notations of Section \ref{bernoulli}, we define $U(n, N, k_0, k_1, \ep, \del)$
to consist of those $S \in \Scal_0$ that satisfy
\begin{enumerate}
\item
$H(\vee_{-k_0}^{N-1}  \hat{S}^{-i} \Pcal_n \, | \vee_{i=-k_1}^{-1} \hat{S}^{-i} \Pcal_n) < N + k_0 + \del$
\item
$H( \vee_{i=0}^{N-1} \hat{S}^{-i} \Pcal_n \, |  \vee_{i= - k_1}^{-k_0}  \hat{S}^{-i} \Pcal_n)  
>  H( \vee_{i=0}^{N-1} \hat{S}^{-i} \Pcal_n) - \ep$.
\end{enumerate}
It is now easy to check that the set 
$U(n, N, k_0, k_1, \ep, \del)$ is open and that the $G_\del$ set
$$
\Scal_1 = \bigcap_{n,l,m,N} \bigcup_{k_0, k_1} U(n, N, k_0, k_1, 1/l, 1/m)
$$
consists of those $S \in \Scal_0$ for which $\hat{S}$ is K.

To see that $\Scal_1$ is not empty when the base $T$ is an arbitrary K-automorphism we
use a result of Parry \cite[Theorem 6]{P} who showed that if
a circle extension of a K-automorphism is weakly mixing then it is K. 
By  a  theorem of Jones and Parry \cite[Theorem 8]{J-P},
a generic circle extension of a weakly mixing system is weakly mixing and we conclude that 
the generic circle extension of the K-automorphism $T$ is K.

Finally, applying again the relative Halmos theorem \cite[Proposition 2.3]{G-W}, we deduce that 
$\Scal_1$ is a dense $G_\del$ subset of $\Scal$, as claimed. 
\end{proof}

\br

In the context of factors of Bernoulli shifts the following question is open.
Can there be a full entropy factor $(Y,S)$ of a Bernoulli shift $(X,T)$ such that the relative product
$X \underset{Y}{\times} X$ is ergodic but not a K-automorphism --- i.e. 
can this relative product have a nontrivial zero entropy factor?
We will next show that for any finite entropy K-automorphism $(X,T)$ the generic extension
$(\hat{X}, \hat{T})$ is such that the relative product $\hat{X} \underset{X}{\times}\hat{X}$ is also K.
We use (a bit modified) notation as in Section \ref{add-0}.

\begin{thm}\label{RI-K}
Let $\Xb = (X,\Xcal,\mu,T)$ be a finite entropy K-automorphism, and $S$ a Rokhlin cocycle
with values in MPT$(I, \la)$, where $I =[0,1]$ and $\la$ is Lebesgue measure on $I$.
We denote by $\hat{S}$ the transformation on the relative independent product $X \times I \times I$ 
defined by 
$$
\hat{S}(x, u, v) = (Tx, S_xu, S_xv), \quad (x,u,v) \in X \times I \times I.
$$
Then for a generic $S \in \Scal$ the transformation $\hat{S}$ is a K-automorphism. 
\end{thm}

\begin{proof}
As usual the proof divides into two parts, showing that our property is a $G_\del$ subset of $\Scal$
and then showing that it is nonempty. 

Now $\Qcal_n$ will denote the product dyadic partition of $I \times I$ into squers of size 
$\frac{1}{2^n} \times \frac{1}{2^n}$ and if $\Rcal$ is a generating partition of $(X,T)$, 
we denote by $\Pcal_n = \Rcal \times \Qcal_n$.
It follows from Section \ref{add-0} that for a dense $G_\del$ susbet $\Scal_0 \subset \Scal$,
the corresponding $\hat{S}$ for $S \in \Scal_0$, has the same entropy as $T$. 
With this change of notation the proof of Theorem  \ref{Kthm} shows that the set of $S \in \Scal_0$
such that $\hat{S}$ is a K-automorphism is a $G_\del$ set.
It remains to show that it is not empty. 

Fix a mixing zero entropy system $\Zb = (Z, \Zcal, \nu, R)$ and an independent partition
$\{C_0, C_1, C_{-1}\}$ of $X$ such that $\mu(C_1) = \mu(C_{-1}) < 1/4$,
define $f(x)$ by setting $f(x) = i$ for $x \in C_i, \ i =0, 1, -1$, so that $\int f \, d\mu =0$, and
define $T_f : X \times Z \times Z \to X \times Z \times Z$ by
$$
T_f (x,z_1, z_2) = (Tx, R^{f(x)} z_1, R^{f(x)} z_2).
$$
An easy version of the proof by Meilijson \cite{Meil}, shows that $T_f$ is ergodic.
Now, according to Rudolph \cite[Corollary 8]{Ru-86},
when $T$ is a K-automorphism and  $R$ (hence also $R \times R$) is mixing, 
then if $T_f$ is ergodic, it must be K.

Another way to see this is as follows. Use Austin's result \cite{Au} to write
the K-system $\Xb$ as a product $X = \Yb \times \Bb$ with $\Yb$ a K-system
and $\Bb$ Bernoulli. Let $\Zb$ be a mixing system and apply Meilijson's 
construction \cite{Meil}, to obtain a skew product extension on $B \times Z$  which is K.
Now the system $\hat{\Xb}$ on $\hat{X} = X \times Z = Y \times B \times Z$ is K and 
the independent relative product
$$
\hat{X} \underset{X}{\times} \hat{X}    =
(Y \times B \times Z) \underset{Y \times B}{\times}(Y \times B \times Z)
\cong Y \times ((B \times Z) \underset{B}{\times}(B \times Z)),
$$ 
is also a K system.

Again an application of the relative Halmos theorem \cite[Proposition 2.3]{G-W}, finishes the proof.
\end{proof}

\br

 Let us recall the definition of relative mixing. 
  
\begin{defn}
A factor map $\pi  :\Xb  \to \Yb$ is called {\em relatively mixing} if for $ f, g \in L_{\infty}(X)$
\begin{equation}\label{rm}
\lim_{n \to \infty} \|\E(T^n f \cdot  g \mid Y)-\E(T^n f\mid Y)\E(g \mid Y) \|_2 =0.
\end{equation}
 \end{defn}
 
When $g$ is $\Ycal$-measurable then 
 $$ 
\E(T^n f \cdot  g \mid Y)-\E(T^n f\mid Y)\E(g \mid Y) = g \E(T^n f \mid Y) - g \E(T^n f \mid Y)  = 0,
$$ 
so we can replace in equation (\ref{rm}),  $g$ by $g - \E(g |Y)$, whose conditional expectation is zero.
Thus an equivalent condition is: for $ f, g \in L_{\infty}(X)$ with  $\E(g \mid Y) =0$,
\begin{equation}\label{rm1}
\lim_{n \to \infty} \|\E(T^n f \cdot  g \mid Y)\|_2 =0.
\end{equation}

 It is well known that the set of mixing transformations is meager (see \cite{H}).  
 In the paper \cite{S} by Mike Schnurr, the following relative version is proved (theorem 6):   
 the set of transformations $T$ acting on the product $X_1\times   X_2$  and leaving the $\sigma$-algebra $X_2$ invariant, in such a way that $T$ is relatively  mixing with respect to $X_2$, is meager.

However, as a consequence of Theorem \ref{RI-K},
we will show that, when $T$ acting on $X$ is K, then the generic extension of $T$ will be relatively mixing
over $X$.

Let us mention that this result sheds some light on
the following old question, originally due to D. Ornstein: 
given a Bernoulli shift $T$, does there exist a factor of $T$ relative to which 
$T$ is weakly mixing but not strongly mixing? 
Now, in view of the following theorem, there is no hope of using Baire category arguments
to resolve this question,
as was done historically in the ``absolute" case.

\begin{thm}\label{K-m}
Let $\Xb =(X, \Xcal,\mu,T)$ be a K-automorphism, then the generic extension of $\Xb$ 
is relatively mixing over $\Xb$.
\end{thm}

\begin{proof}
This result is a consequence of Theorem \ref{RI-K}. and the following lemma:

\begin{lem}
Let $\Xb$ be ergodic and $\Yb$ be a factor of $\Xb$ with factor map $\pi : \Xb \to \Yb$. 
Then the following are equivalent:
\begin{enumerate}
\item
$\Xb$ is a relatively mixing extension of $\Yb$.
\item
In the relatively independent product $X\underset {Y}{\times} X$, the Koopman operator restricted to $L_2(Y)^{\perp}$ is mixing.
\end{enumerate}
\end{lem}

\begin{proof}
Let $X = Y \times Z$ be the Rohlin representation of $\Xb$ over $\Yb$ and
let $\mu = \int_Y \mu_y\,d\nu(y)$ be the disintegration of $\mu$ over $\nu$.
Let 
$$
W := X \underset{Y}{\times} X = \{(x_1, x_2) \in X \times X \colon \pi(x_1) = \pi(x_2)\},
$$
and let the relative product measure $\la$, supported on $W$, be given by
$$
\la = \int_Y \mu_y \times \mu_y \, d\nu(y).
$$
Thus for $F \in L_2(\la)$
$$
\int F(x_1,x_2) \, d\la(x_1,x_2) = \int_Y \left( \int_{X \times X} F(x_1,x_2) \, 
d\mu_y(x_1)\, d\mu_y(x_2)\right) \, d\nu(y).
$$

Note that a function $F(x_1, x_2) \in L_2(\la)$ is in ${L_2(Y)}^\perp$ \  iff \ 
$$
\E(F  \mid  Y)(y) = \int_X F(x_1,x_2) \, d\mu_y(x_1) d\mu_y(x_2),
$$
is $\nu$  a.e.  $0$.

\br

(1)$ \Rightarrow $ (2).\   
Given $f_1, f_2, g_1, g_2 \in L_\infty(X)$, let
$$
f_1 \otimes f_2 (x_1, x_2) = f_1(x_1)f_2(x_2), \quad g_1 \otimes g_2 (x_1, x_2) = g_1(x_1)g_2(x_2),
$$
and note that, e.g.,  $\E( f_1 \otimes f_2 \mid Y) = \E(f_1 \mid Y) \E(f_2 \mid Y)$.
Let
\begin{gather*}
F = f_1 \otimes f_2 - \E(f_1 \mid Y) \E(f_2 \mid Y)\\
G = g_1 \otimes g_2 - \E(g_1 \mid Y) \E(g_2 \mid Y).
\end{gather*}
Since linear combinations of functions of the form $f_1 \otimes f_2$ are dense in $L_2(\la)$,
mixing in $L_2(Y)^\perp$ will follow from 
\begin{equation}\label{m-perp}
\left| \int_W T^n F \cdot G \, d\la \right| \to 0
\end{equation}
(where we write $T^nF$ for the diagonal action of $T$ on $X \times X$).
Expanding the left hand side of the last formula we get
$$
\int_Y \E\large([T^n(f_1\otimes f_2) - T^n(\E(f_1 \mid Y) \E(f_2 \mid Y))]
[g_1\otimes g_2 - \E(g_1 \mid Y) \E(g_2 \mid Y)] \mid Y \large)\, d\nu.
$$
There are four terms in the product $T^n F \cdot G$.
The first one is :
\begin{gather*}
\int_W T^n(f_1\otimes f_2)\cdot g_1\otimes g_2\, d\la =\\
\int_Y \left( \int_X T^n f_1(x_1) g_1(x_1) \, d\mu_y(x_1)\right)
 \left( \int_X T^n f_2(x_2) g_2(x_2) \, d\mu_y(x_1)\right)\, d\nu(y).
\end{gather*}
By (1), i.e. by relative mixing (\ref{rm}) (and boundedness of the functions), for large $n$,
we can  replace, with only a small error, the expression
$\int_X T^n f_1(x_1) g_1(x_1) \, d\mu_y(x_1)$ by $\E(T^n f_1 \mid Y)\E(g_1 \mid Y)(y)$.

Similarly the other three terms can be replaced by the same expression, one with plus sign, two
with minus sign, and thus the total is indeed close to zero, as claimed.

\br

  (2)  $\Rightarrow$  (1) 
 Assuming $\E(g \mid Y)=0$, we need to show that equation (\ref{rm1}) holds.
 Now under this assumption, with $f = f_1 = f_2$ and $g = g_1 = g_2$, 
 real valued functions,
 equation (\ref{m-perp}) reads  
\begin{gather*}
 \int_W T^nf(x_1) T^n  f (x_2) \cdot g(x_1)  g(x_2)\, d\la = 
 \int_Y \E(T^n f \cdot g \mid Y)\cdot  \E f(T^n  f \cdot g  \mid Y) \, d\nu(y) \\
 = \parallel  \E(T^nf  \cdot g \mid Y) \parallel_2^2   \rightarrow 0.
\end{gather*}
\end{proof}

Applying the lemma to the relative independent product in Theorem \ref{RI-K}, and recalling the fact that
every K-automorphism is mixing, these arguments complete the proof of Theorem \ref{K-m}.
\end{proof}

\br

\section{A generic extension of a LB system is LB}\label{LB}

\begin{thm}
For any fixed loosely Bernoulli transformation $T$, the generic extension is loosely Bernoulli.
\end{thm}

\begin{proof}
We recall first one of the definitions of Loosely Bernoulli (LB) transformations.
There are two kinds, one with zero entropy and the other having positive entropy.

For the latter; $(X, \Xcal, \mu, T)$ is called {\em loosely Bernoulli}
if there is some set $A \in \Xcal$ such that the induced transformation $T_A$ is isomorphic 
to a Bernoulli shift.
The zero entropy loosely Bernoulli are defined in a similar fashion except that now 
$T_A$ is required to be isomorphic to an irrational rotation. For the basic theory and facts that
we will use see \cite{ORW}.

There is a characterization of process that define LB transformations similar to the VWB condition,
but with the $\bar{d}_n$ metric replaced by the $\bar{f}_n$ metric which we proceed to define.

For two words $u, v \in \{1,2,\dots,a\}^n$ we define
$$
\bar{f}_n(u,v) = 1 - \frac{k}{n},
$$
where $k$ is the maximal integer for which
we can find subsequences $0 \leq i_1 < i_2 < \cdots i_k \leq n-1$ and
$0 \leq j_1 < j_2 < \cdots j_k \leq n-1$, with 
$$
u(i_r) = v(j_r), \quad 1 \leq r \leq k.
$$
This defines a metric on words and using it instead of the normalized Hamming metric we also
define the $\bar{f}_n$ metric on probability distributions on $\{1, 2, \dots,a\}^n$.

\begin{defn}
A finite measurable partition $\Pcal$ of $X$ is called {\em very loosely Bernoulli} (VLB for short) if
for every $\ep>0$, there is a positive integer $N$ such that for all $k \geq 1$, there is a
collection $\Gcal_k$ of atoms of the partition $\bigvee_{i = -k}^{-1} T^{-i}\Pcal$ such that
\begin{enumerate}
\item
$\mu(\bigcup \{A \in \Gcal_k\}) > 1 -\ep$.
\item
$\bar{f}_N\left(\vee_{i = 0}^{N -1} T^{-i}\Pcal, \vee_{i = 0}^{N -1} T^{-i}\Pcal \rest A \right) < \ep$,
for every atom $A$ of $\Gcal_k$.
\end{enumerate}
Here $\bar{f}_N$ is the normalized $\bar{f}_N$ distance of distributions, 
and $\bigvee_{i = 0}^{N -1} T^{-i}\Pcal \rest A$ is the 
conditional distribution of $\bigvee_{i = 0}^{N -1} T^{-i}\Pcal$
restricted to the atom $A$ of the partition $\bigvee_{i = -k}^{-1} T^{-i}\Pcal$.
\end{defn}

\br

For zero entropy LB there is a simpler formulation as follows.

\begin{defn}
Let $\Pcal = \{P_1, \dots, P_a\}$ be a measurable partition of $X$.
Given  $\al \in \{1, 2, \dots,a\}^n$, we set $[\al] = \cap_{i=0}^{N-1} T^{-i}P_{\al_i}$.
A zero entropy process, defined by $\Pcal = \{P_1, \dots, P_a\}$, is {\em very loosely Bernoulli}
if given $\ep >0$ there is an $N$ and a 
a subset $G_N \subset \{1, 2, \dots,a\}^N$
such that
\begin{enumerate}
\item
$\sum \{\mu([\al]) : \al \in G_N\} > 1 -\ep$.
\item
for all $\al, \al' \in G_N$,  $\bar{f}_N (\al, \al') < \ep$.
\end{enumerate}
\end{defn}

Once again it suffices to show that there is refining sequence of partitions
$$
\Pcal_1 \prec \cdots \prec \Pcal_n \prec \Pcal_{n+1} \prec \cdots
$$
such that the corresponding algebras $\hat{\Pcal}_n$ satisfy
$\bigvee_{n \in \N} \hat{\Pcal}_n = \Xcal$, and such that for each $n$, the process $(T, \Pcal_n)$ is VLB.

In zero entropy, to see that a generic extension of
a LB $T$ is LB, we keep notations from previuos sections and define, for each $n, N, \ep$
$$
U(n, N, \ep) =
\{S \in \Scal_0 :  {\text{the process}}\  (\hat{S}, \Pcal_n) \ {\text{satisfies (1) and (2)}}\}.
$$
These are open sets and the set
$$
\bigcap_{n,k} \bigcup_{N=1}^\infty U(n, N, 1/k)
$$
is a $G_\del$ set and consists of LB transformations.
To see that it is not empty we can use the fact that for any LB transformation $T$ 
with zero entropy,  a compact abelian group extension is also LB \cite[Theorem 7.3]{ORW}.

\br

For the positive entropy case we repeat exactly the same argument as in Section \ref{bernoulli}
with $\bar{f}_N$ replacing $\bar{d}_N$.
The fact that there are relative zero entropy extensions of any LB transformation 
of positive entropy again follows from the fact that, in positive entropy,
any ergodic isometric extension of a LB transformation is LB \cite[Corollary 8]{Ru-79}.
\end{proof}

\br

\end{document}